\documentclass{article}
\usepackage{subfigure}
\usepackage[square,numbers]{natbib}

\pdfoutput=1

\usepackage{amssymb}
\usepackage{amsmath}
\usepackage{amsthm}
\usepackage{pifont}
\usepackage{graphicx}
\usepackage[left=3cm]{geometry}

\newcommand{\rationals}{\ensuremath{\mathbb Q}}

\usepackage{xcolor}

\def\cal{\mathcal}

\newtheorem{theorem}{Theorem}[section]
\newtheorem{lemma}[theorem]{Lemma}
\newtheorem{proposition}[theorem]{Proposition}

\newenvironment{definition}[1][Definition]{\begin{trivlist}
  \item[\hskip \labelsep {\bfseries #1}]}{\end{trivlist}}

\newcommand{\onev}{(1)}
\newcommand{\twov}{(2)}
\newcommand{\threev}{(3)}
\newcommand{\xv}{(i)}

\newcommand{\Intra}{\text{Intra}}
\newcommand{\Inter}{\text{Inter}}

\title{Distances in \\ random Apollonian network structures} 

\author{Olivier Bodini, Alexis Darrasse and Mich{\`e}le Soria}
\date{16 November 2007}

\begin{document}

\maketitle
\begin{abstract}
In this paper, we study the distribution of distances in random Apollonian
network structures (RANS), a family of graphs which has a one-to-one
correspondence with planar ternary trees.  Using multivariate generating functions
that express all information on distances, and singularity analysis for
evaluating the coefficients of these functions, we describe the distribution of
distances to an outermost vertex, and show that the average value of the
distance between any pair of vertices in a RANS of order $n$ is asymptotically
$\sqrt n$.
\end{abstract}

\section{Introduction}\label{sec:intro}

Many graph models have been recently introduced for representing the structure
and dynamics of real-life networks (see e.g.\ \citep{newman2006sad}).  Their
adequacy to data can be measured by comparing some properties of graphs,
especially the degree distribution of the vertices, which is related to
scale-free properties, and properties related to the ``small world'' effect,
such as distance between pairs of vertices and grouping in clusters.

The random Apollonian networks (RAN) proposed by \citep{zhou2005mpn} provides a
very interesting model, with a power-law degree distribution, a mean distance of
logarithmic order and a large clustering coefficient.  We introduced in
\citep{darrasse2007ddr} a modified version, random Apollonian network structures
(RANS), which preserves the interesting properties of real graphs concerning
degree distribution (a power-law with an exponential cut-off) and large
clustering.  This paper is devoted to the analysis of distances in RANS, which
is showed to be of square root order: we first characterize the distances from
one special vertex to all the other vertices of the graph, and then work on the
distances between pairs of vertices.

A RANS can be seen as a certain type of triangulation of a triangle, and the
study of RANS relies on the bijection with planar ternary trees (see
figure~\ref{fig:rans}).  From this bijection we can express the enumerative
generating function for RANS, and use multivariate functions for marking
several distance parameters.  Moreover the asymptotic values of the quantities
under consideration can be dealt with using singularity analysis (according to
methods developed in \citep{flajolet2007ac}).

We are interested in two types of parameters measuring distance, and develop two
methods to handle them. We first attack the distances between a special vertex
(an outermost vertex A of the RANS) and all the other vertices. The method
is built on computing a generating function with infinitely many variables, that
contains all informations concerning distances from A to the other vertices.
Distribution analysis is based on the study of partial derivatives of this
multivariate series, which correspond to the series counting the number of
vertices at a certain distance from A. These series all express in terms of the
generating function for RANS and asymptotic analysis gives a distribution
with a mean value of order $\sqrt{3\pi n}/11$.

The second study addresses the total distance between all pairs of vertices.
We exhibit a generating function in four variables that expresses
simultaneously distances from one, two or three outermost vertices. This
generating function has a nice recursive definition, due to the symmetries of
the problem. It contains all information to compute the total distance between
pairs of vertices.  Geometrical considerations splits this total distance in two
parts, depending on whether a path between two vertices spans over disjoints
sub-RANS or not.  The resulting mean distance between two vertices is of
order $2\sqrt{3\pi n}/11$.

This paper divides in four sections: this introduction, followed by a section
that recalls the definition of random Apollonian network structures, the
bijection with ternary trees, and the result for degree distribution.  Section 3
describes the distribution of distances from an outermost vertex and section 4
is dedicated to the study of the total distance between all pairs of vertices.

\section{Random Apollonian network structures}\label{sec:rans}
The recursive definition of RANS shows a one-to-one correspondence with ternary
trees. The degree distribution, which is a power law with an exponential
cut-off, was studied in \citep{darrasse2007ddr} by considering bivariate series
marking the corresponding parameter in trees.

\subsection{Bijection with ternary trees}\label{subsec:bij}
A random Apollonian network structure (RANS) $R$ is recursively defined as:
either an empty triangle or a triangle $T$ split in three parts, by placing a
vertex $v$ inside $T$ and connect it to the three vertices of the triangle;
each sub-triangle being substituted by a RANS (see figure~\ref{fig:rans}).

The vertices of $T$ will be called the \emph{outermost vertices} of $R$ (noted
$\cal{O}(R)$); and vertex $v$ will be called the \emph{center} of $R$. We will
note $\cal{R}$ the class of all RANS.

The \textit{order} of the empty RANS is zero and the order of a non-empty RANS
is one plus the sum of the orders of the three sub-RANS\@.

\begin{proposition}\label{prop:bij} \citep{darrasse2007ddr}
  There is a bijection between random Apollonian network structures of order
  $N$ and rooted plane ternary trees of size $N$ (with $N$ internal nodes).
\end{proposition}

In planar ternary trees, the linear ordering of siblings is relevant.
This order is carried over to triangles: naming $\{O_1,O_2,O_3\}$
the vertices of $\cal{O}(R)$, imposes a linear ordering on the sub-RANS
($\{S_1,S_2,S_3\} = \cal{S}(R)$): $S_i$ will be the one not containing $O_i$.
Recursively replacing the missing outermost vertex by the center of $R$
preserves the order in sub-RANS\@.

The generating function for ternary trees  $T(z) = \sum T_{N}z^{N}$ satisfies
the functional equation $T(z) = 1 + z T^3(z)$, whose solution can be analysed
locally through singularity analysis.  $T(z)$ has radius of convergence $\rho =
4/27$ and singular value $\tau = 3/2$; and the singular expansion of $T(z )$
near $\rho$ is
\begin{equation}\label{eqn:texpansion}
  T(z) = \frac{3}{2} - \frac{\sqrt3}{2} \sqrt {1-z/\rho} + \frac{2}{3}
  (1-z/\rho) - \frac{35\sqrt{3}}{108}(1-z/\rho)^{3/2} + O\left((1-z/\rho)^{5/2}\right).
\end{equation}
Thus the asymptotic form of the coefficients: $T_N\sim c \rho^{-N} N^{-3/2}$,
with $c= \sqrt{3}/4\sqrt{\pi}$.

The derivative $T'(z) = \frac{T^3(z)}{1-3zT^2(z)}$ will also appear in the
computations below. The leading term in its singular expansion is $\frac{\sqrt
3}{4 \rho} (1-{z}/{\rho})^{-1/2}$, thus a coefficient  $ T'_{N}$ of asymptotic
order $\frac{27\sqrt{3}}{16\sqrt{\pi}}\rho ^{-N} N^{-1/2}$.

\subsection{Degree distribution}\label{subsec:degree}
The degree distribution in random Apollonian network structures follows a power
law with an exponential cutoff. This is obtained by analysing the degree of the
center of a RANS (which corresponds to the size of a binary subtree at the root
of the corresponding ternary tree), and propagating this study to the whole of
the sub-RANS\@.
 
The bivariate generating function marking the degree of the center is $D_g(z,u) =
z u^3 T^3(z,u)$, where $T(z,u)$ is the bivariate generating function for ternary
trees with $u$ marking the size of the underlying binary subtree, which is also
the degree of an outermost vertex: 
\begin{equation}\label{eqn:Tdegbiv}
  T(z,u) = 1 + uz T(z) T^2(z,u).
\end{equation}

\begin{theorem}\label{th:degree} \citep{darrasse2007ddr}
  The degree distribution in random Apollonian network structures follows  a
  power law with an exponential cutoff:
  $\Pr (D_g = k) \sim C\, \beta^k \ k^{-3/2}$, with
  $\beta = \frac{8}{9}$.
\end{theorem}

\begin{figure}
  \centering
  \includegraphics{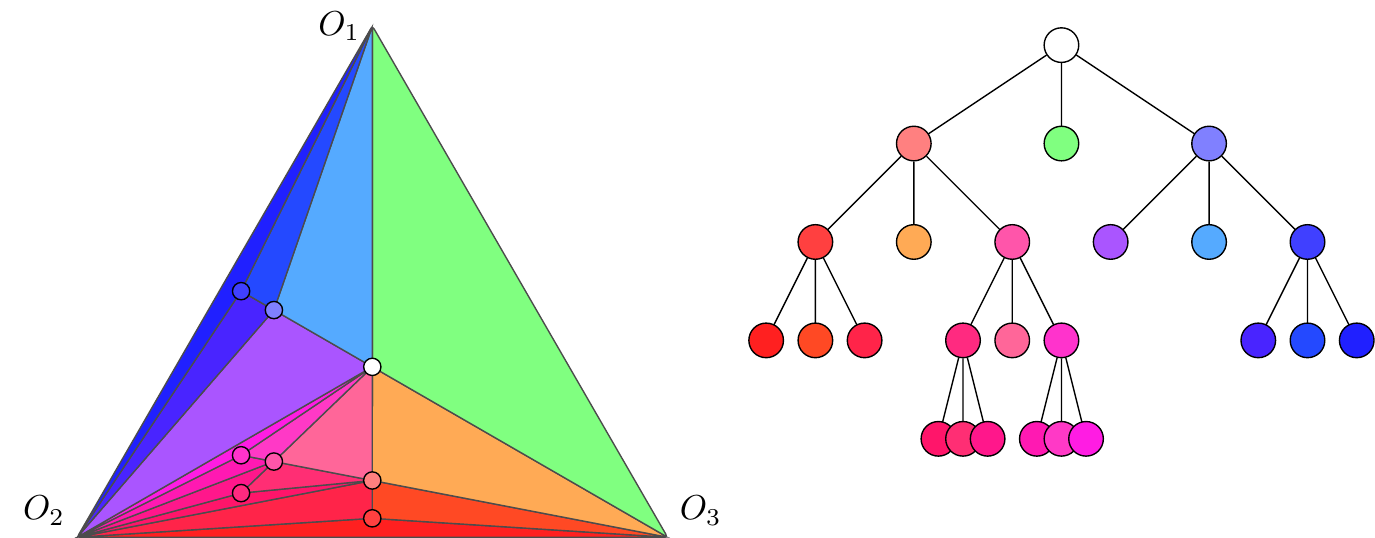}
  \caption{A random Apollonian network and its corresponding ternary tree}
  \label{fig:rans}
\end{figure}

\newpage

\section{Distance from an outermost vertex}\label{sec:distA}

This section is devoted to computing the distribution of the distances from a
fixed specific outermost vertex. We introduce a generating function with
infinitely many variables, each variable $u_{i}$ marking the number of vertices
at distance $i$ from the outermost vertex. Relying on the symmetries of the
problem and the recursive nature of RANS, we are able to express and study this
generating function.

The interest of this analysis is not only to find, with a different method,
some results of the following section; but moreover this study can be adapted to
compute the distribution of distances to \emph{any distinguished vertex} in a RANS,
which may be considered as a more realistic parameter.

\subsection{Multivariate generating function}\label{subsec:mgf}

Due to the recursive nature of RANS, we often have to consider a RANS $R$ as
having an \textit{environment}, that is a bigger RANS containing $R$.  Given a
RANS $R$, the distance of any of its vertices to a vertex $v$ in the
environment of $R$ is determined by the three distances of the elements of
$\cal{O}(R)$ to $v$.

Since the outermost vertices of $R$ form a clique, their distances to any
vertex cannot differ by more than one. This observation allows us to reduce our
study to a few cases.  First we work modulo a translation and restrict
ourselves to the case when the three distances to $\cal{O}(R)$ are either $1$
or $0$.  Second we can work modulo a permutation and restrict ourselves to only
three cases: $(0,1,1)$, $(0,0,1)$ and $(0,0,0)$, illustrated in
figure~\ref{fig:types}.  These three cases actually correspond to labelling the
internal vertices of $R$ by their distances either to one (out of three), or to
two (out of three) or to all three outermost vertices.

\begin{definition}\label{def:label}
  The \textit{$\delta_\onev$-labeling} of $R\in\cal{R}$ consists in putting on each
  vertex a label corresponding to its distance from $O_1(R)$ (or equivalently to
  one of any $O_i(R)$):
  \begin{list}{$\bullet$}{\setlength{\topsep}{0pt}\setlength{\itemsep}{0pt}
                          \setlength{\parsep}{0pt}}
    \item the outermost vertices $O_1,O_2,O_3$ respectively receive labels $0$,
    $1$ and $1$;
    \item the center of $R'\in\cal{S}(R)$ is labeled by 1 plus the minimum of the
    labels of $\cal{O}(R')$.
  \end{list}
\end{definition}

\begin{definition}\label{def:type}
  Let's define the \textit{type of $R\in\cal{R}$} as the set of labels of
  $\cal{O}(R)$.
  We say that
  \begin{list}{$\bullet$}{\setlength{\topsep}{0pt}\setlength{\itemsep}{0pt}
                          \setlength{\parsep}{0pt}}
    \item two RANS types are \textit{equivalent} iff they have the same type up to a
    permutation.
    \item two RANS types are \textit{translated} by $\theta$ iff their labellings are the
    same up to a translation $\theta$.
  \end{list}
\end{definition}

\begin{figure}
  \centering
  \includegraphics{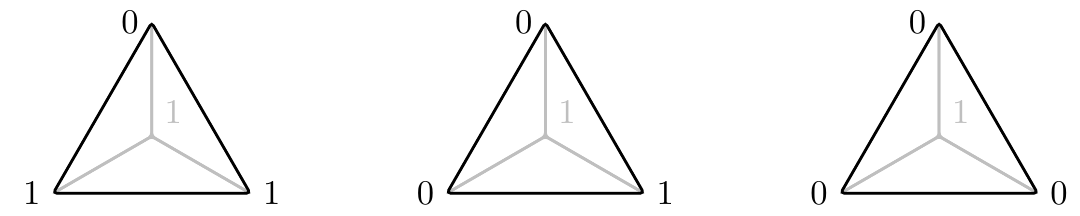}
  \caption{RANS of type $(0,1,1), (0,0,1), (0,0,0)$ and their sub-RANS\@.}
  \label{fig:types}
\end{figure}

In a RANS $R$ of type $(0,1,1)$ the center gets label 1. Thus $R$ is equivalent
to $S_2(R)$ and $S_3(R)$, but it is not equivalent to $S_1(R)$, which is of type
$(1,1,1)$: if a vertex is at distance $d$ from $O_1(S_1(R))$, its distance from
$O_1(R)$ is $d+1$.

This remark leads to the bivariate generating function for RANS marked with
vertices at distance 1 from $O_1(R)$:
\begin{equation}\label{eqn:T1biv}
  T_1(z,u_1) = 1 + {\color{black!20!brown}z u_1} {\color{teal}T_1^2(z,u_1)} {\color{violet} T(z)}.
\end{equation}
This follows the recursive definition of RANS, noticing first that the center
is {\color{black!20!brown}at distance $1$ from $O_1(R)$} and second that
{\color{teal}the configuration is the same} in $S_2(R)$ and $S_3(R)$, whereas
in $S_1(R)$ {\color{violet}there is no vertex at distance $1$ from $O_1(R)$}.
Note that (\ref{eqn:T1biv}) is obviously the same generating function as
(\ref{eqn:Tdegbiv}), the degree of a vertex  being exactly the number of
vertices at distance $1$ from this vertex.

The problem of marking \textit{both} vertices at distance 1 and 2 from
$O_1(R)$ is treated in the same way: the configuration of $R$, {\color{teal}of
type $(0,1,1)$, recursively occurs} in $S_2(R)$ and $S_3(R)$, which are of the
same type. But the case of $S_1(R)$, {\color{violet}of type $(1,1,1)$}, is a
little more tricky and requires a deeper decomposition. If $S_1(R)$ is not
empty, its center is {\color{black!20!brown}at distance 2 from $O_1(R)$}, and
its three sub-RANS $\cal{S}(S_1(R))$ are equivalent, {\color{purple}of type
$(1,1,2)$}: either they are empty, or their center is
{\color{black!20!brown}at distance 2 from $O_1(R)$}, one sub-RANS is
{\color{purple}of type $(1,1,2)$}, equivalent to $\cal{S}(S_1(R))$ and the two
others are {\color{green!30!black}of type $(1,2,2)$}, translated by 1 with
$R$, which means that the number of their vertices at distance 2 from $O_1(R)$
in is the same as the number of vertices at distance 1 from $O_1(R)$ in $R$.

This decomposition leads to the following functional equations, with $u_j$
marking vertices at distance $j$ from $O_1(R)$:
\begin{align*}
  T_2(z,u_1,u_2) & = 1 + z u_1 {\color{teal}T_2^2(z,u_1,u_2)} {\color{violet}F(z,u_1,u_2)} \notag \\
  F(z,u_1,u_2) & = 1 + {\color{black!20!brown}zu_2} {\color{purple}G^3(z,u_1,u_2)}, \\
  G(z,u_1,u_2) & = 1 + {\color{black!20!brown}zu_2} {\color{purple}G(z,u_1,u_2)}  {\color{green!30!black}T_1^2(z,u_2)}. \notag
\end{align*}

It is easy to show that the same equations hold for $d\ge 3$, when considering
multivariate generating functions $T(z,u_1,u_2,\ldots,u_d)$, with $u_j$ marking
vertices at distance $j$ from $O_1(R)$, and we get the following result.

\begin{proposition}\label{proposition:recurSG}
  Let $r_{n, k_1, \ldots, k_d }$ be the number of RANS of order $n$ with $k_j$
  vertices at distance $j$ from $O_1$. Then $r_{n, k_1, \ldots, k_d}$ is the
  coefficient of $u_1^{k_1}u_2^{k_2}\ldots u_d^{k_d}z^n$ in the multivariate
  series $T_d(z,u_1, \ldots, u_d)$, where the series $T_d$ satisfy the
  recurrence relations:
  \begin{equation*}
    T_1(z,u_1) = 1 + z u_1 T_1^2(z,u_1)T_0(z) \qquad\hbox{with} \qquad T_0(z) =T(z)\\
  \end{equation*}
  and for $d\ge 2$,
  \begin{equation*}
    T_d(z,u_1,\ldots,u_d) = 1 + z u_1 T_d^2(z,u_1,\ldots,u_d) \left(1+ z u_2\, \frac{1}{(1-z u_2 T_{d-1}^2(z,u_2,\ldots, u_{d-1}))^3}\right).
  \end{equation*}
\end{proposition}

The sequence $T_d(z,u_1,\ldots,u_d) $ converges to a function
$T_{\infty}(z,u_1,\ldots,u_i,\ldots)$ which contains all information concerning
distances from vertex $O_1$:

\begin{list}{$\bullet$}{\setlength{\topsep}{0pt}\setlength{\itemsep}{0pt}
                        \setlength{\parsep}{0pt}}
  \item The enumerative series for the number of vertices at distance $i$ from $O_1$, over all RANS, is
  \begin{equation*}
    D_i(z) =
    \left.\frac{\partial}{\partial u_i}T_{\infty}(z,u_1,\ldots,u_i,\ldots) \right|_{u_j=1, \forall j} = \sum_n k_i\, r_{n,k_i} z^n.
  \end{equation*}
  \item The asymptotic of the total distance from $O_1$ expresses as
    \begin{equation*}
      \left.\frac{\partial}{\partial u}D(z,u)\right|_{u=1}, \qquad \text{where } D(z,u) = \sum_{i=1}^{\infty} D_i(z)u^i.
    \end{equation*}
\end{list}

The aim of the next paragraph is to evaluate these quantities.

\subsection{Distribution analysis}\label{subsec:analysis}
Generating functions counting the number of vertices at distance $i$ from
$O_{1}$ express as rational functions in $z$ and $T(z)$, and have a singular
behaviour similar to $T(z)$: radius of convergence $\rho$, and singular expansion
of the square-root type.

\begin{lemma}\label{lemma:solRec}
  The sequence of  enumerative series for the number of vertices at distance
  $i$ from $O_1$ is:
  \begin{align*}
    D_1(z) & = {z T^3(z)}/{(1 - 2 z T^2(z))}, \\
    D_2(z) & = H(z, T(z))\times{(1+ 2 z^2 T^4(z))}/(6z T(z) (1 - 2 z T^2(z))) \\
    \hbox{and for } \,\, i\ge 2 \qquad
        D_{i+1}(z) & = H^{i-1}(z, T(z)) \times D_2(z),
  \end{align*}
  where $H(z, T(z))$ is a rational function in $z$ and $T(z)$, that has radius
  of convergence $\rho=4/27$, and a singular expansion $H(z) = 1 -
  \frac{11}{\sqrt 3} \sqrt{1-z/\rho} + \frac{2}{3} (1-z/\rho) +
  (1-z/\rho)^{3/2} + O((1-z/\rho)^2)$.
\end{lemma}

\begin{proof} 
  From (\ref{proposition:recurSG}), it is easy to compute the expressions of
  $D_1(z)$ and $D_2(z)$, and show that $D_{i+1}(z) = H(z, T(z)) \times D_i(z)$,
  with $H(z, T(z)) = 6z^2(T(z)-1)T(z)/(1 - 3z - zT(z) - zT^2(z) + 2z^2T^2(z))$\footnote{Since
  $\rationals[z,T(z)]/\langle T(z) - 1 - T^3(z)\rangle$ is a $\rationals(z)$-vector
  space with dimension three, all rational functions in $z$ and $T(z)$ that
  appear in this paper can be expressed in a canonical form.  However we didn't
  use it since it usually hides the combinatorial interpretation of the
  generating functions under consideration.}.  The
  singular expansion comes from expressing $z$ as $(T(z)-1)/T^3(z)$ and
  plugging in $H$ the singular expansion (\ref{eqn:texpansion}) of $T(z)$. A
  full expansion of $T(z)$ yields a full expansion for $H(z)$, the first terms
  of which are given in the lemma.
\end{proof}

The full singular expansion of  $D_i(z)$ can be derived from its expression in
terms of $H$ and $D_2$. Thus the proportion of vertices at distance $i$ from
$O_1$, that is $\frac{1}{nT_n}[z^n]D_i(z)$ can be evaluated. We have no closed
form to express these quantities (it is work in progress), but plotting from
experimental results obtained on a sample of randomly generated RANS, shows that
the distances from $O_1$ follow the distribution shown in
figure~\ref{fig:profile}. 

\begin{figure}[h]
  \centering  
  \includegraphics{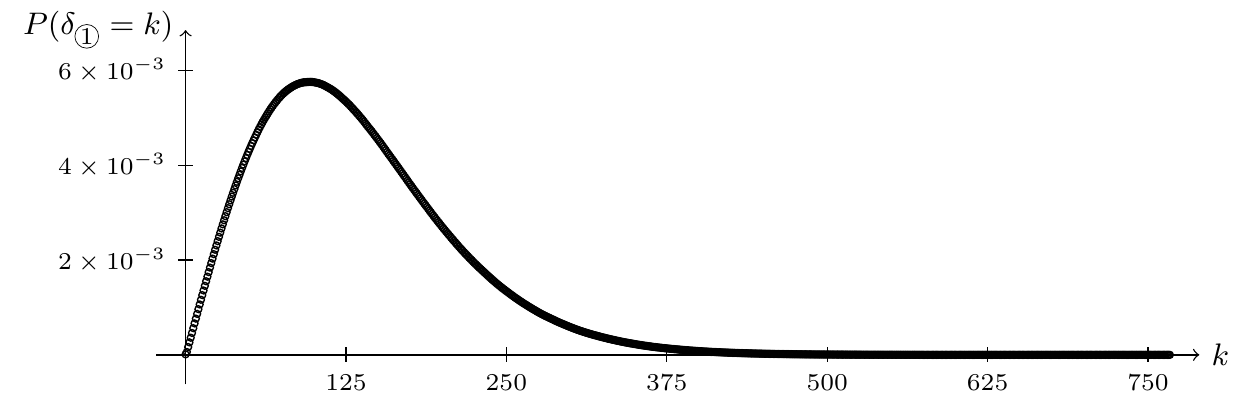}
  \caption{Distances from $O_1$. The points represent the
  experimental data on RANS of orders 1000--1400.}
  \label{fig:profile}
\end{figure}

The average distance from $O_1$ is also obtainable by derivation of $D(z,u) =
\sum D_i(z)u^i$. From lemma~\ref{lemma:solRec} there is a closed form for
$D(z,u)$, and derivation leads (fortunately) to the same series as the one
obtained for $\Delta_\onev(z)$, in section~\ref{subsec:topogf}.

This series has a singular expansion around $\rho$ with first term
$\frac{3}{44} (1-z/\rho)^{-2}$, so for the mean distance 
\begin{equation*}
  \frac{1}{n T_{n}} [z^{n}] \left.\frac {\partial}{\partial u}  D(z,u) \right|_{u=1} = \frac{\sqrt{3\pi n}}{11} \left(1+ O(\frac{1}{n})\right).
\end{equation*}
We thus conclude this section with the following proposition:
\begin{proposition}\label{proposition:meanO1}
  In a RANS of order $n$, the average distance from $O_1$ is of order $c\sqrt{n}$,
  with $c=\sqrt{3\pi}/11$.
\end{proposition}

\vfill
\section{Total distance between pairs of vertices}\label{sec:totdist}
\begin{figure}[b]
  \centering  
  \includegraphics{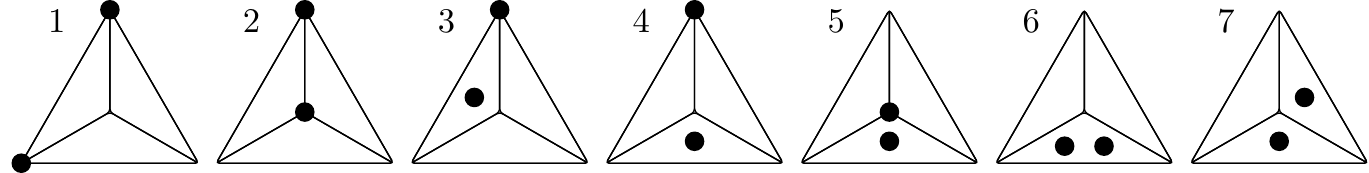}
  \caption{All possible configurations for pairs of vertices in $R\in\cal{R}$.
  For each pair in $\cal{C}(R)$ we can always find a unique sub-RANS of $R$ such
  that the two vertices are in one of configurations 2, 4 or 7.  The other cases
  reduce to one of these: case 1 reduces to 2 by looking at the RANS containing
  $R$, case 3 leads to 2, 3 or 4 by looking in the sub-RANS containing the two
  vertices, and such is also the case for case 5.  Case 6 is amenable to any of
  5, 6 or 7 by looking at the sub-RANS containing the two vertices.}
  \label{fig:pairs}
\end{figure}

In this section, we are interested in computing the total distance of every
pair of vertices in a RANS of order $n$, and will show that the mean value of
this quantity is still of order $\sqrt n$.

We call $\cal{C}(R)$ the set of pairs of vertices (we call pair a set of size
two) in $R\in\cal{R}$, excluding pairs where both vertices are in $\cal{O}(R)$.

The enumerative generating function for the total distance between pairs in
$\cal{C}(R)$ is
\begin{equation*}
  G(z) = \sum_{R\in\cal{R}}\sum_{(v,w)\in \cal{C}(R)} dist(v,w)\,\, z^{|R|}.
\end{equation*}

$\cal{C}(R)$ splits into two parts
\begin{itemize}
  \item the pairs $(v,w)$ such that they are both \textit{internal} vertices
  of the smallest sub-RANS of $R$ that contains both of them,
  corresponding to case $7$ in figure~\ref{fig:pairs}.
  
  We will note them $\Inter(R)$ and their contribution to the
  total distance will be called \textit{interdistance}.

  \item the others, which can also be defined as the pairs $(v,w)$ such that
  there exists a sub-RANS $S$ of $R$ with $v$ an outermost vertex of $S$ and $w$
  an internal vertex of $S$,
  corresponding to cases $2$, $3$ and $4$ in figure~\ref{fig:pairs}.
  
  We will note them $\Intra(R)$ and their contribution to the
  total distance will be called \textit{intradistance}.
\end{itemize}

\newpage

\textbf{Remark:} $\cal{C}(R)$ has $n(n-1)/2 + 3n$ elements: each pair of
internal vertices is counted only once, and the $3n$ term takes into account all
pairs made of one internal vertex and one outermost vertex. Among all these
pairs, an amount of order $n\sqrt{n}$ belongs to $\Intra(R)$ and the rest is in
$\Inter(R)$.  As we will show, the total distance of pairs in $\Intra(R)$ is of
order $n^2$ and the total distance of pairs in $\Inter(R)$ is of order
$n^2\sqrt{n}$. We can thus say that the interdistance gives the dominant term of
the total distance in RANS, which is of order $\sqrt{n}$.

We introduce in the following subsection a new generating function which serves
as a basis for the computations of all the quantities that are needed. Then we
calculate the intradistance followed by the interdistance. Putting everything
together gives the following result.

\begin{theorem}
  The mean distance in a RANS of order $n$ is asymptotically equivalent to
  $C\sqrt{n}$, with $C=\sqrt{3\pi}/22$.
\end{theorem}

\subsection{Topological generating function }\label{subsec:topogf}

Given $R\in\cal{R}$, the distances of inner vertices to $\cal{O}(R)$ are denoted by
the three following parameters:
\begin{equation*}
  \Delta_\onev(R) = \sum_{x\in R} d(x,O_1(R)),\:
  \Delta_\twov(R) = \sum_{x\in R} d(x,\{O_1(R),O_2(R)\}) \:\text{and}\:
  \Delta_\threev(R) = \sum_{x\in R} d(x,\cal{O}(R)).
\end{equation*}
Notice that $\Delta_\onev(R)$ is the sum of all the labels in the
$\delta_\onev$-labeling of $R$, based on RANS of type $(0,1,1)$ ---see
(\ref{def:label}). Similarly, $\Delta_\twov(R)$ is the sum of all the labels in
the $\delta_\twov$-labeling of $R$, that starts with a RANS of type $(0,0,1)$;
and $\Delta_\threev(R)$ is the sum of all the labels in the
$\delta_\threev$-labeling of $R$, starting with a RANS of type $(0,0,0)$.

In the following generating function, parameter $\Delta_\xv$ is marked by variable $d_\xv$:
\begin{equation*}
  \Delta(z,d_\onev,d_\twov,d_\threev) =
  \sum_{R\in\cal{R}} d_\onev^{\Delta_\onev(R)} d_\twov^{\Delta_\twov(R)} d_\threev^{\Delta_\threev(R)} z^{|R|} =
  \sum_{n,i,j,k=0}^\infty \alpha_{n,i,j,k} \, d_\onev^i d_\twov^j d_\threev^k z^n,
\end{equation*}
where $\alpha_{n,i,j,k}$ is the number of RANS of order $n$ (\textit{i.e.} with
$n$ internal points), and respective values $i, j, k$ for parameters
$\Delta_\onev, \Delta_\twov, \Delta_\threev$. This generating function is
called the \textit{topological generating function} since it expresses the
distances according to the three different topological types of RANS\@. 

\begin{proposition}\label{proposition:rectopo}
  The topological generating function satisfies the recursive equation
  \begin{align*}
    \Delta(z,d_\onev,d_\twov,d_\threev) = 1 + z d_\onev d_\twov d_\threev
             & \times \Delta(zd_\onev,d_\twov,d_\threev,d_\onev) \\
             & \times \Delta(z,d_\onev, d_\twov d_\threev, 1) \\
             & \times \Delta(z,d_\onev d_\twov, d_\threev, 1).
  \end{align*}
\end{proposition}

\begin{proof}
  Let's follow the recursive definition of RANS $R$.  If $R$ is empty the
  contribution to the series is 1.  Otherwise it has a center, which is at
  distance 1 from each of the outermost vertices (hence the factor $z d_\onev
  d_\twov d_\threev$) and the contributions come from the 3 sub-RANS.
    
  Factor $\Delta(zd_\onev,d_\twov,d_\threev,d_\onev)$ comes from $S_1(R)$.
  Suppose $S_1(R)$ has, by itself, a generating series
  $\Delta(z,d_\onev,d_\twov,d_\threev)$, that corresponds to the three
  different labellings, with types $(0,1,1)$, $(0,0,1)$ and $(0,0,0)$. When it
  is considered as embedded as the first sub-RANS of $R$, the top most vertex
  has label 1 instead of 0, so that the three different labellings now start
  with types $(1,1,1)$, $(1,0,1)$ and $(1,0,0)$. Thus variable $d_\onev$
  transforms into $d_\twov$, variable $d_\twov$ transforms into $d_\threev$,
  and variable $d_\threev$ transforms into $d_\onev$ with a 1-translation.

  Factor $\Delta(z,d_\onev,d_\twov d_\threev,1)$ comes from $S_2(R)$.  Suppose
  $S_2(R)$ had, by itself, a generating series $\Delta(z, d_\onev, d_\twov,
  d_\threev)$, corresponding to the three different types $(0,1,1)$, $(0,0,1)$
  and $(0,0,0)$. When it is considered as embedded as the second sub-RANS of
  $R$, $O_2(S_2(R))$ has to have label 1, so that the three different
  labellings now start with types $(0,1,1)$,  $(0,1,1)$ and $(0,1,0)$. Thus
  variables $d_\twov$ and $d_\threev$ transform into $d_\twov$, variable
  $d_\onev$ stays $d_\onev$, and nothing transforms to $d_\threev$.

  Factor $\Delta(z,d_\onev d_\twov,d_\threev,1)$ comes from $S_3(R)$, and the
  proof is equivalent.
\end{proof}

The series of cumulated distances from $A$ is obtained by derivation
\begin{equation*}
  \Delta_\onev(z) = \sum_{R\in\cal{R}}\Delta_\onev(R)z^{|R|}
                  = \left.\frac{\partial}{\partial d_\onev} \Delta(z,d_\onev,1,1) \right|_{d_\onev=1}
\end{equation*}
and the same holds for the two other cases.

\begin{proposition}
  The distance generating functions $\Delta_\xv(z)$ have the following expressions:
  \begin{align*}
    \Delta_\onev(z) & = zT(z)^3(1 - 2zT(z)^2 + z^2T(z)^4 - 6z^3T(z)^6)/Q(z,T(z)) \\
    \Delta_\twov(z) & = zT(z)^3(1 - 3zT(z)^2 + 4z^2T(z)^4 - 6z^3T(z)^6)/Q(z,T(z)) \\
    \Delta_\threev(z) & = zT(z)^3(1 - 3zT(z)^2 + 2z^2T(z)^4)/Q(z,T(z)), \\
    & \text{where } Q(z,T(z)) = (1 + 2z^2T(z)^4)(1 - 3zT(z)^2)^2.
  \end{align*}
  Each $\Delta_\xv(z)$ has radius of convergence $\rho$ and a singular expansion of
  the form:
  \begin{equation*}
    \Delta_\xv(z) = 3/(44(1-z/\rho)) + O(1-z/\rho)^{-1/2}.
  \end{equation*}
\end{proposition}

\begin{proof}
  The distance generating functions $\Delta_\xv(z)$ satisfy the system of equations:
  \begin{align*}
    & \left\{
    \begin{array}{rcl}
      \Delta_\onev(z) & = & z T^3(z) + 2 z \Delta_\onev(z) T^2(z) + z T^2(z)(z T'(z)+ \Delta_\onev(z)) \\
      \Delta_\twov(z) & = & z T^3(z) + 2 z \Delta_\onev(z) T^2(z) + z \Delta_\threev(z) T^2(z) \\
      \Delta_\threev(z) & = & 3 z \Delta_\twov(z)T^2(z) + z T^3(z) \\
    \end{array}
    \right. \\
    \text{\rm where } & T'(z) = T^3(z) / (1- 3 z T^2(z)).  \notag
  \end{align*}
  The resolution of the system shows that each $\Delta_\xv(z)$ has a dominant term
  that expresses as $T'^2(z)$ with the same constant factor, thus a pole in
  $z=\rho$. The singular expansions only differ on their second term.
\end{proof}

\subsection{Intradistance}\label{subsec:intra}

We first consider the pairs $(v,w)$ such that there exists a sub-RANS $S$ of $R$
with $v$ an outermost and $w$ an internal vertex of $S$. There may be many embedded
sub-RANS $S$ and we focus on the smallest one, $S_0$. In $S_0$, vertex $v$ is
outermost (e.g.\ $O_1$) and $w$ is either the center of $S_0$ or in the sub-RANS
opposite to $v$ (e.g.\ $S_1(S_0)$) (cf.\ cases $2$ and $4$ in
figure~\ref{fig:pairs}).

We will first study the pairs $\Intra_1(R)$, for which $S_0 = R$, and then
recursively extend the computation to the rest of the intradistance.

\begin{lemma}
  The generating function for the total distance of pairs in $\Intra_1(R)$,
  satisfies
  \begin{equation*}
    \delta(z) = 3 T(z) + 3 z T^2(z) \Delta_3(z) + 3 z^2 T^2(z) T'(z).
  \end{equation*}
\end{lemma}

\begin{proof}
  The distance of $\Intra_1(R)$ is made of two categories of distances:
  \begin{list}{$\bullet$}{\setlength{\topsep}{0pt}\setlength{\itemsep}{0pt}
                          \setlength{\parsep}{0pt}}
    \item from the center of $R$ to each of the outermost vertices
    $\cal{O}(R)$,
    \item from each outermost vertex $O_i(R)$ to all the internal vertices of
    its opposed sub-RANS, $S_i(R)$.
  \end{list}

  The distance from an outermost vertex $O_i(R)$ to an internal vertex $w$ of
  $S_i(R)$ is $1 + d(w, \cal{O}(S_i(R)))$.  Thus, the distance from an
  outermost vertex $O_i(R)$ to all the internal vertices of $S_i(R)$ is
  $|S_i(R)| + \Delta_\threev(S_i(R))$.  Taking into account all three sub-RANS
  of $R$ we have
  \vspace{-0.5em}
  \begin{equation*}
    \delta(z) = \sum_{R\in\cal{R}}\bigg(3 + \sum_{S\in\cal{S}(R)}\left(\Delta_\threev(S) + |S|\right)\bigg)\,\,z^{|R|},
  \end{equation*}
  thus the expression of the generating function as stated in the lemma.
\end{proof}

\begin{theorem}\label{th:intra}
  The generating function for intradistances in a RANS is $\Intra(z) =
  \delta(z)/(1 - 3zT^2(z))$ and the total distance between pairs of vertices in
  $\Intra(R)$, for $R\in\cal{R}_n$, is asymptotically $\frac{1}{44}n^2$.
\end{theorem}

\begin{proof}
  The total intradistance is obtained by recursively computing intradistances at
  any level of the RANS\@. The effect of this recursion process, akin to recursive
  decent in subtrees of ternary trees, is to multiply the generating function by
  $T'(z)/T^3(z)$, that is $1/(1-3zT^2(z))$. The dominant term in the singular
  expansion of $\Intra(z)$ thus is $3z\Delta_\onev(z)T'^2(z)/T(z)$. The total distance in
  $\Intra(R)$ is obtained by evaluating $[z^n]\Intra(z)/T_n$.
\end{proof}

\subsection{Interdistance}\label{subsec:inter}

We now consider the pairs $(v,w)$ such that they are both \textit{internal}
vertices of the smallest sub-RANS $S$ of $R$ that contains both of them.

Since $S$ is minimal by inclusion, $v$ and $w$ are in different sub-RANS of
$S$, which we will call $S_v$ and $S_w$. The shortest path from $v$ to $w$
passes through at least one of the two vertices of $\cal{O}(S_v) \cap
\cal{O}(S_w) = \text{Frontier}(S_v,S_w)$.  We can thus decompose this path in
three sub-paths: from $v$ to $\text{Frontier}(S_v,S_w)$, from
$\text{Frontier}(S_v,S_w)$ to $w$ and, if these two sub-paths are disjoint, an
one-edge path along $\text{Frontier}(S_v,S_w)$.  We will call this last part
the \textit{f-edge}.  This decomposition is illustrated in
figure~\ref{fig:inter}.

\begin{figure}[b]
  \centering
  \includegraphics{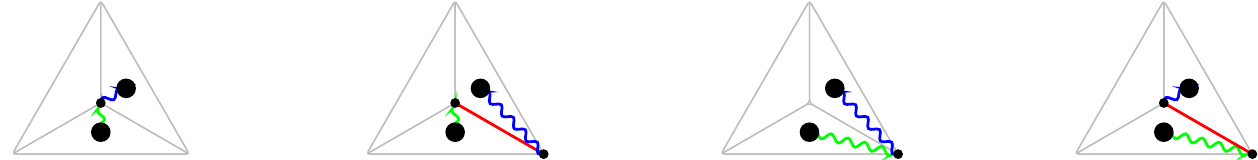}
  \caption{The four possible scenarios for paths between pairs of vertices in
  $\Inter(R)$. The three colors correspond to the three sub-paths, the red one
  being the \textit{f-edge}.}
  \label{fig:inter}
\end{figure}

We will first compute a lower bound $\Inter^{-}(R)$ of the interdistance by
neglecting the f-edges. This lower bound gives a total distance on pairs of
$\Inter(R)$, with $R\in\cal{R}_n$ order $n^2\sqrt{n}$. We can also take an upper
bound $\Inter^{+}(R)$ by forcing every path to pass from the center of $R$ and
this still gives a contribution of order $n^2\sqrt{n}$ with the same factor.
Counting the exact number of f-edges allows us to compute the following terms of
the interdistance.

\subsubsection{Lower bound and upper bound}

As for the intradistance we first compute a lower (resp.\ upper) bound on
interdistances at the topmost level $\Inter_1(R)$ (i.e.\ $S_v,S_w \in
\cal{S}(R)$) and extend it recursively for the whole RANS\@.

\begin{lemma}
  The generating function for lower bound (resp.\ upper bound) of the total
  distance of pairs in $\Inter_1(R)$, $\gamma^{-}(z)$ (resp.\ $\gamma^{+}(z)$),
  satisfies
  \begin{equation*}
    \gamma^{-}(z) = 6 z^2 T(z) T'(z) \Delta_\twov(z) \qquad\qquad
    \gamma^{+}(z) = 6 z^2 T(z) T'(z) \Delta_\onev(z).
  \end{equation*}
\end{lemma}

\begin{proof}
  At level one, for each sub-RANS, the contribution to the interdistance is the
  total length of the sub-paths contained in this sub-RANS\@.  Thus for each
  vertex $v$, situated in a sub-RANS $S_1(R)$, we will count its distance to
  the frontier, multiplied by the number of vertices in $S_2(R)$ and $S_3(R)$.
  The lower bound is obtained by adding all these values (parameter
  $\Delta_\twov$), and for the upper bound we consider that the frontier is
  reduced to only one of its two points (parameter $\Delta_\onev$). Thus the
  expression of the generating function $\gamma^{-}(z)$ (and $\gamma^{+}$ is
  obtained by replacing $\Delta_\twov$ by $\Delta_\onev$):
  \begin{equation*}
    \gamma^{-}(z) = 3\sum_{R\in\cal{R}}\Delta_\twov(S_1(R))\times(|S_2(R)|+|S_3(R)|)\,z^{|T|}.
  \end{equation*}
  \vspace{-3em}

\end{proof}

\begin{theorem}\label{th:interb}
  The generating function for the lower bound (resp.\ upper bound) of
  interdistances in a RANS is
  \begin{equation*}
    \Inter^{-}(z) = \frac{\gamma^{-}(z)}{1 - 3zT^2(z)} \qquad\qquad
    \Inter^{+}(z) = \frac{\gamma^{+}(z)}{1 - 3zT^2(z)}
  \end{equation*}
  and in both cases the total distance between pairs of vertices in
  $\Inter(R)$, for $R\in\cal{R}_n$, is asymptotically $Cn^2\sqrt{n}$ with
  $C=\sqrt{3\pi}/11$.
\end{theorem}

\begin{proof}
  The proof is similar to the proof of theorem~\ref{th:intra}. The generating
  functions $\Inter^{-}(z)$ and $\Inter^{+}(z)$ have both a dominant term in
  $6\Delta_\xv(z)z^2T'^2(z)/T^2(z)$.
\end{proof}

\subsubsection{Exact computation}

To know whether the path between two vertices $(v,w) \in \Inter(R)$ contains a
f-edge it helps to know the distances from $v$ and $w$ to each of the two
vertices $(f_1,f_2)$ of $\text{Frontier}(S_v,S_w)$.

We distinguish two cases:
\begin{list}{$\bullet$}{\setlength{\topsep}{0pt}\setlength{\itemsep}{0pt}
                        \setlength{\parsep}{0pt}}
  \item The first one when either $d(v,f_1) = d(v,f_2)$ (we then
  say that $v \in \cal{E}(R)$) or $d(w,f_1) = d(w,f_2)$.  In
  this case the path between $v$ and $w$ does not contain a f-edge.
  \item Otherwise either $v$ and $w$ are both closer to the same vertex of the
  frontier or each one of them is closer to a different vertex of the frontier.
  In the first case there is no f-edge on the path between $v$ and $w$ while on
  the second case there is one.
\end{list}

These two cases being equiprobable, thanks to the symmetry, it is sufficient to
calculate the number of pairs $(v,w)\in\Inter(R)$ for which $v,w\notin
\cal{E}(R)$, and the number of f-edges will be the half of this quantity.

\begin{lemma}
  The generating function for the number of f-edges in the total distance of
  $\Inter_1(R)$ is
  \begin{equation*}
    \phi(z) = \frac{3}{2} \left(z^3T(z)T'^2(z) - 2z^2T(z)T'(z)E(z) + zT(z)E^2(z)\right)
  \end{equation*}
  where $E(z) = \sum_{R\in\cal{R}}|\cal{E}(R)|z^{|R|}$.
\end{lemma}

\begin{proof}
  The number of pairs in $\Inter_1(R)$ for which $v,w\notin\cal{E}(R)$ has
  generating function
  \begin{equation*}
    \phi(z) = 3 \sum_{R\in\cal{R}} (|S_1| - E(S_1))(|S_2| - E(S_2))z^{|R|}.
  \end{equation*}
  \vspace{-3em}

\end{proof}

\paragraph{Calculating $E(z)$.} We use a bivariate generating function
$T_e(z,u)$ for RANS marked with vertices at equal distance from $O_1$ and
$O_2$.  This will be defined by a system of thirteen equations in the same
spirit as in section~\ref{subsec:mgf}. The analysis of this system is too long
to be included in this abstract, but it leads to
\begin{equation*}
  E(z) = \frac{3 z T'(z)}{2 T^2(z)} \, \left(z T'(z) - \frac{T(z)(2 - 4 T(z) + 3 T^2(z) - T^3(z))}{(2 T(z)-3)(3 T^2(z) - 4 T(z)+2}\right)^2
\end{equation*}
and a singular expansion around $\rho$ which is equivalent to
$\frac{5\sqrt{3}}{44}\times(1 - z/\rho)^{-1/2}$.

\begin{theorem}
  The generating function for the number of f-edges in a RANS is $F(z) =
  \phi(z)/(1 - 3zT^2(z))$ and the total number of f-edges in R, for
  $R\in\cal{R}_n$, is asymptotically $\frac{9\sqrt{\pi}}{242}n^2$.
\end{theorem}

\begin{proof}
  The proof is similar to theorem~\ref{th:intra}. But in this case, each term
  of $\phi(z)$ gives a part of the dominant contribution.
\end{proof}

\subsection{Conclusion}

Summing the contribution of intradistances and interdistances the enumerating
generating function for the total distance between pairs of vertices expresses
as:
\begin{equation*}
  G(z) = \Intra(z) + \Inter^{-}(z) + F(z),
\end{equation*}
which has a closed form expression as a rational function in terms of $z$ and
$T(z)$.

We made an exhaustive study of the different parts of $G(z)$.  The contribution
coming from intradistances happens to be of smaller order ($n^2$) than the
contribution coming from interdistances ($n^2\sqrt{n}$).  In the computation of
interdistances, we first considered approximations that give a lower and an
upper bound with the same dominant term ($\frac{\sqrt{3\pi}}{11}n^2\sqrt{n}$)
that is a mean distance in $\frac{2\sqrt{3\pi}}{11}\sqrt{n}$. The study of
f-edges provides an exact computation of the total distance. With this
contribution of f-edges it is possible to express the second term in the
asymptotic expression of the total distance. Moreover, relying on full singular
expansion of all series under consideration, it is possible to give a full
asymptotic expansion of the total distance.

\end{document}